\documentclass[12pt]{article}
\usepackage{psfrag,amsmath,amssymb,latexsym,epic, eepicemu, epsfig,
 float, enumerate}
\usepackage{amscd }
\usepackage{amsfonts}
\usepackage{graphicx}
\usepackage{epsfig,psfrag,amsmath,amssymb,latexsym}
\usepackage{amscd}
\usepackage[dvips]{color}
\usepackage{amsfonts}
\usepackage{graphics}
\usepackage{graphicx}
\usepackage{float}
\usepackage[T1]{fontenc}
\newtheorem{theorem}{Theorem}[section]

\newtheorem{corollary}{Corollary}[section]

\newtheorem{lemma}{Lemma}[section]

\newtheorem{proposition}{Proposition}[section]

\newenvironment{proof}[1][Proof]{\textbf{#1.} }{\ \rule{0.5em}{0.5em}}
\newcommand{\td}{\buildrel{D}\over{=}}
\newcommand{\X}{{\cal X}_o}
\newcommand{\e}{\epsilon}
\newcommand{\R}{{\bar R}}
\def\P{{\bf P}}
\newcommand{\V}{\tilde{V}}
\numberwithin{equation}{section}
\def\td{\buildrel{D}\over{=}}
\def\X{{\cal X}_o}
\def\e{\epsilon}
\begin{document}

\title{Asymptotic risks of Viterbi segmentation}

\author{Kristi Kuljus, Jüri Lember\thanks{Estonian science foundation grant no
7553}}

\maketitle

{\small \vspace{0cm} \hbox{\hspace{1.1 cm}\vbox{\noindent\hsize=8cm Swedish University of Agricultural Sciences\\
Centre of Biostochastics\\
901 83 Ume{\aa}, Sweden\\
E-mail: Kristi.Kuljus@sekon.slu.se}\vbox{\noindent \hsize=8cm  University of Tartu\\
Institute of Mathematical Statistics \\
Liivi 2-513 50409, Tartu,
Estonia\\
E-mail: jyril@ut.ee}} \vskip 1cm \hfill} \vspace{0.5cm}\noindent
\abstract{\noindent We consider the maximum likelihood (Viterbi)
alignment of a hidden Markov model (HMM). In an HMM, the underlying
Markov chain is usually hidden and the Viterbi alignment is often
used as the estimate of it. This approach will be referred to as the
Viterbi segmentation. The goodness of the Viterbi segmentation can
be measured by several risks. In this paper, we prove the existence
of asymptotic risks. Being independent of data, the asymptotic risks
can be  considered as the characteristics of the model that
illustrate the long-run behavior of the Viterbi segmentation.}

\vskip 1\baselineskip\noindent {\bf Keywords:} hidden Markov model,
Viterbi alignment, segmentation.

\section{Introduction}\label{Intro}
\noindent The present paper deals with asymptotics of the Viterbi
segmentation. Before we can present main results, we introduce the
segmentation problem and different risks for measuring goodness of
segmentations.
\subsection{Notation}
\noindent Let $Y=\{Y_t \}_{t=-\infty}^{\infty}$ be a double-sided
stationary MC with states $S=\{1,\ldots,|S|\}$ and irreducible
aperiodic transition matrix $\big(P(i,j)\big)$. Let
$X=\{X_t\}_{t=-\infty}^{\infty}$ be a double-sided process such
that: 1) given $\{Y_t\}$  the random variables $\{X_t\}$ are
conditionally independent; 2) the distribution of $X_j$ depends on
$\{Y_t \}$ only through $Y_j$. The process $X$ is sometimes called a
{\it hidden Markov process} (HMP) and the pair $(Y,X)$ is referred
to as a {\it hidden Markov model} (HMM). The name is motivated by
the assumption that the process $Y$, which is sometimes called the
{\it regime}, is non-observable. The distributions $P_s:=\P(X_1\in
\cdot|Y_1=s)$ are called {\it emission distributions}. We shall
assume that the emission distributions are defined on a measurable
space  $({\cal X},{\cal B})$, where ${\cal X}$ is usually
$\mathbb{R}^d$ and ${\cal B}$ is the Borel $\sigma$-algebra. Without
loss of generality we shall assume that the measures $P_s$ have
densities $f_s$ with respect to some reference measure $\mu$. Our
notation differs from the one used in the HMM literature, where
usually $X$ stands for the regime and $Y$ for the observations.
Since our study is mainly motivated by statistical learning, we
would like to be consistent with the notation used there and keep
$X$ for observations and $Y$ for latent variables.
\\
HMMs are widely used in various fields of applications, including
speech recognition \cite{rabiner, jelinek}, bioinformatics
\cite{koski, BioHMM2}, language processing  \cite{ochney}, image
analysis \cite{gray2} and many others. For general overview about
HMMs, we refer to \cite{HMMraamat} and \cite{HMP}.
\\\\
Given a set ${\mathcal A}$ and integers $m$ and $n$, $m<n$, we shall
denote any $(n-m+1)$-dimensional vector with all the components in
${\mathcal A}$ by $a_m^n:=(a_m,\ldots,a_n)$. When $m=1$, it will be
often dropped from the notation and we write $a^n\in {\mathcal
A}^n$.
\subsection{Segmentation and risks}
\noindent The {\it segmentation} problem consists of estimating the
unobserved realization of the underlying Markov chain $Y_1,\ldots,
Y_n$ given $n$ observations $x^n=(x_1,\ldots ,x_n)$ from a hidden
Markov model. Formally, we are looking for a mapping $g:{\cal X}^n
\to S^n$ called a {\it classifier}, that maps every sequence of
observations into a state sequence (see \cite{seg} for details). For
finding the best $g$, it is natural to set to every state sequence
$s^n\in S^n$ into correspondence a measure of goodness of $s^n$,
referred to as the {\it risk of $s^n$}. Let us denote the risk of
$s^n$ for a given $x^n$ by $R(s^n|x^n)$. The solution of the
segmentation problem is then a state sequence with minimum risk. In
the framework of pattern recognition theory the risk is specified
via a {\it loss function} $L: S^n\times S^n \to [0,\infty],$ where
$L(a^n,b^n)$ measures the loss when the actual state sequence is
$a^n$ and the estimated sequence is $b^n$. For any state sequence
$s^n\in S^n$ the  risk is then
\[\label{risk}
R(s^n|x^n):=E[L(Y^n,s^n)|X^n=x^n]=\sum_{a^n\in
S^n}L(a^n,s^n)\P(Y^n=a^n|X^n=x^n).
\]
One common loss function is the so-called {\it symmetric loss}
$L_{\infty}$ defined as
\[\label{symm-n}
L_{\infty}(a^n,b^n)=\left\{
               \begin{array}{ll}
                 1, & \hbox{if $a^n\ne b^n$;} \\
                 0, & \hbox{if $a^n=b^n$.}
               \end{array}
             \right.
\]
We shall denote the corresponding risk by $R_{\infty}$. With this
loss, $R_{\infty}(s^n|x^n)=\P(Y^n\ne s^n|X^n=x^n)$, thus the
minimizer of $R_{\infty}(\cdot|x^n)$ is a sequence with maximum
posterior probability, called the {\it Viterbi alignment}. The name
is inherited from the dynamic programming algorithm (Viterbi
algorithm) used for finding it. Let $v$ stand for the Viterbi
alignment, i.e.~$v(x^n)=\arg\max_{s^n}p(s^n|x^n)$, where
$p(s^n|x^n)=\P(Y^n=s^n|X^n=x^n)$. Obviously, the Viterbi alignment
is not necessarily unique. The Viterbi alignment minimizes also the
following risk:
\begin{equation}\label{loglikerisk}
\R_{\infty}(s^n|x^n):=-{1\over n}\ln p(s^n|x^n).
\end{equation} The log-likelihood based risk (\ref{loglikerisk})
is often preferable to use since it allows various generalizations,
see (\ref{hybrid}).
Another common classifier is based on the pointwise loss function
\begin{equation}\label{point-loss}L_1(a^n,b^n)={1\over
n}\sum_{t=1}^nl(a_t,b_t),\end{equation} where $l(a_t,b_t)\geq 0$ is
the loss of classifying the $t$-th symbol $a_t$ as $b_t$. Typically,
for every state $s$, $l(s,s)=0$. Let us denote the corresponding
risk by $R_1(s^n|x^n)$:
$$R_1(s^n|x^n)={1\over n}\sum_{t=1}^nR^t_1(s_t|x^n),$$
where $R^t_1(s|x^n):=\sum_{a\in S} l(a,s)p_t(a|x^n)$ and
$p_t(a|x^n):=\P(Y_t=a|X^n=x^n).$ Most frequently $l(s,s')=I_{\{s\ne
s'\}}$, and then $R_1(s^n|x^n)$ just counts the expected number of
misclassified symbols given that the data are $x^n$ and the sequence
$s^n$ is used for segmentation. For that $l$,
\begin{equation}\label{r1symm}
R_1(s^n|x^n)=1-{1\over n}\sum_{t=1}^n p_t(s_t|x^n).
\end{equation}
The minimizer of (\ref{r1symm})  over all the possible state
sequences is called the {\it pointwise maximum a posteriori} (PMAP)
alignment. The Viterbi and the PMAP-classifier  -- the so-called
standard classifiers -- are by far the two most popular classifiers
used in practice.
\\
We shall also consider the risk
$$\R_1(s^n|x^n):=-{1\over n}\sum_{t=1}^n\ln p_t(s_t|x^n).$$
The risks $R_1$ and $\R_1$ are closely related. Minimizing
(\ref{r1symm}) over all possible state sequences is clearly
equivalent to minimizing $\R_1$, but this is not necessarily so for
restricted minimization.  The importance of $\R_1$ and $\R_{\infty}$
becomes apparent in \cite{seg}, where  the following penalized
$\R_1$-risk is considered:
\begin{equation}\label{hybrid}
\R_{C}(s^n|x^n):=\R_1(s^n|x^n)+C\R_{\infty}(s^n|x^n).
\end{equation}
Here $C\geq 0$ is a given regularization constant. The risk $\R_{C}$
naturally interpolates between the two standard alignments: for
$C=0$ the minimizer of (\ref{hybrid}) is the PMAP-alignment, and it
is not hard to see that for $C$ big enough the minimizer of
(\ref{hybrid}) is the Viterbi alignment. Obviously, the likelihood
of the minimizer of (\ref{hybrid}) increases with $C$ as well as the
$\R_1$-risk. In \cite{seg} it is  shown that
 minimizing the risk $\R_C$ for an integer $C$ is closely related to
maximizing the expected number of correctly estimated tuples of
$C+1$ adjacent states. In \cite{seg} it is also shown that
minimization of $\R_{C}(s^n|x^n)$ as well as of
$R_1(s^n|x^n)+C\R_{\infty}(s^n|x^n)$ can be carried out by a dynamic
programming algorithm that is similar to the Viterbi algorithm and
easy to implement.
\subsection{Organization of the paper and main results}
\noindent The quantity $R(g,x^n):=R(g(x^n)|x^n)$ measures the
goodness of a classifier $g$, when it is applied to the observations
$x^n$. When $g$ is optimal in the sense of risk, then
$R(g,x^n)=\min_{s^n}R(s^n|x^n)=:R(x^n)$. We are interested in the
random variables $R(g,X^n)$. The present paper deals mostly with
convergence of the risks of Viterbi alignments. The results are all
largely based on the regenerativity of the Viterbi process
$\{V_t\}_{t=1}^{\infty}$, which is an $S$-valued stochastic process
that is in a sense the limit of the random vectors $v(X^n)$ as $n$
grows. The existence of the Viterbi process is crucial and not
obvious, our analysis is based on the results in \cite{IEEE,
AVT4,K2}, where the Viterbi process is constructed piecewise.
\\
In this paper we shall show that under fairly general assumptions on
an HMM, the random variables $R_1(v,X^n)$, $\R_1(v,X^n)$ as well as
$\R_{\infty}(X^n):=\R_{\infty}(v,X^n)$  all converge to constant
limits almost surely. These convergences are stated in Theorems
\ref{main1}, \ref{main2} and \ref{main3}, which are the main results
of the paper.
%
%
The limits -- {\it asymptotic risks} -- are constants that all
depend on the model and characterize the goodness of the
segmentation based on the Viterbi alignment. If, for example, $R_1$
is the limit of $R_1(v,X^n)$ and $R_1^*$ is the limit of $R_1(X^n)$,
then the difference $R_1-R^*_1$ shows how  well the Viterbi
alignment performs the segmentation in the long run in the sense of
$R_1$-risk
 in comparison to the best possible alignment.
If $R_1$-risk is defined as in (\ref{r1symm}), then for $n$ big
enough the Viterbi alignment makes approximatively $nR_1$
classification errors, while the best alignment in this case -- the
PMAP-alignment -- makes approximatively $nR^*_1$ errors. Since the
model is known, the asymptotic risks could in principle be found
theoretically, but the convergence theorems show that they could
also be found by simulations.
\\
The results concerning the construction of the Viterbi process are
introduced in Subsection \ref{subsec:viterbi}. The piecewise
construction under general assumptions is rather technical (see
\cite{IEEE,K2}). However, when it is performed, the regenerativity
of the Viterbi process as well as the ergodicity of the double-sided
Viterbi process easily follow. The references to necessary results
from the theory of regenerative processes are given in Subsection
\ref{subsec:reg}.
\\
Section \ref{sec:r1} deals with the convergence of the $R_1$-risk.
We prove that $R_1(v,X^n)$ converges to a constant $R_1$ almost
surely.
Section \ref{sec:logr1} proves the convergence of the $\R_1$-risk
for the Viterbi and PMAP-alignment. Since the regenerativity of the
PMAP-process which is the analogue of the Viterbi process for the
PMAP-alignment, is not proved, the regenerativity-based methods
cannot be used for the long-run analysis of PMAP-alignments.
However, as shown in \cite{PMAP}, the convergence of the $R_1$-risk
of the PMAP-alignment can be proved with a completely different
method based on the exponential forgetting or smoothing
probabilities. The exponential forgetting inequalities are
introduced in Subsection \ref{subsec:exp} and in Section
\ref{sec:logr1} we show that they imply also the convergence of the
$\R_1$-risk of the PMAP-alignment.
In Section \ref{sec:loglike}, the convergence of the log-likelihood
or $\R_{\infty}$-risk is proved.
\\
There is no universal method known yet to prove the convergence of
general risks and every optimal alignment needs a special treatment.
For example, the convergence of
$\R_C(X^n)=\min_{s^n}\R_{C}(s^n|X^n)$ (as well as of several other
more general risks introduced in \cite{seg})   has not yet been
proved, although it is reasonable to conjecture that it holds.
Moreover, we conjecture that the dynamic programming algorithm for
finding the minimizer of $\R_C$-risk together with the exponential
smoothing could be used to find the $\R_C$-optimal alignment process
piecewise. If this is true, then the alignment process is
regenerative and the results and methods in the present paper can be
applied to many other optimal alignments.
\section{Preliminary results}
\subsection{Regenerativity}\label{subsec:reg}
\noindent We are following the coupling  approach developed by
Thorisson in \cite{thorisson}. One of the main instruments we are
going to use is that any regenerative process can be successfully
coupled with a stationary and ergodic regenerative process (Theorem
\ref{key}). With a successful coupling, a general pathwise limit
theorem for the Viterbi alignment (Theorem \ref{koond-r}) can be
proved. This is the main preliminary result and it can be used for
many other purposes besides proving the convergence of risks.
\\
Let $Z=\{Z_t\}_{t=1}^{\infty}$ in $(\Omega,{\cal F},\P)$ be a ${\cal
Z}:=\mathbb{R}^{d}$-valued classical regenerative  process with
respect to the renewal process $S=\{S_t \}_{t=0}^{\infty}$ (see,
e.g.~Chapter 10 in \cite{thorisson}). Following the notation in
\cite{thorisson}, we shall denote the regenerative process by
$(Z,S)$. Let $T_1:=S_1-S_0$. The regenerative process $(Z,S)$ is
{\it positive recurrent} if $ET_1<\infty$ and {\it aperiodic} if
$T_1$ is aperiodic, i.e.~$\P(T_1\in a\mathbb{N})<1$ for every $a>1$.
A pair $(Z',S')$ is a {\it version} of the regenerative process
$(Z,S)$ if it is also regenerative and
$\theta_{S_0}(Z,S)\buildrel{D}\over{=}\theta_{S'_0}(Z',S'),$ where
$\theta_t$ is a shift operator:
$\theta_t(x_1,x_2,\ldots)=(x_{t+1},x_{t+2},\ldots)$, and  $\td$
means equal in law. The version $(Z^o,S^o):=\theta_{S_0}(Z,S)$ of
$(Z,S)$ is a {\it zero-delayed} regenerative process.
Thus, $S^o_0=T_1$. Recall that $(Z,S)$ is stationary if
$\theta_t(Z,S)$ has the same distribution as $(Z,S)$. If $(Z,S)$ is
positive recurrent regenerative, then there exists a stationary
version $(Z^*,S^*)$ of this process such that the distribution of
the delay length $S_0^*$ is given by
$$\P(S_0^*=k)={1\over ET_1}\P(T_1>k),\quad k\geq 0,$$
and for every $\sigma({\cal Z}^{\infty})$-measurable function $g:
{\cal Z}^{\infty}\to \mathbb{R}$ the following inequality holds:
\begin{equation}\label{jaotus}
Eg(Z^*_1,Z^*_2,\ldots)={1\over
ET_1}E\big[\sum_{t=0}^{T_1-1}g(\theta_t(Z^o))\big],
\end{equation}
see, e.g.~Theorem 2.1 and 2.2 of Chapter 10 in \cite{thorisson} or
Theorem 6.1 in \cite{Kalashnikov}.
\\\\
Recall that a sub-$\sigma$-algebra of ${\cal F}$ is called {\it
trivial} if its elements have probability 1 or 0. In the following
we consider two $\sigma$-algebras: the tail-$\sigma$-algebra ${\cal
T}:=\cap_{t=1}^{\infty}\theta_t^{-1}(\sigma({\cal Z}^{\infty}))$ and
the $\sigma$-algebra of shift-invariant sets ${\cal I}:=\{A\in
\sigma({\cal Z}^{\infty}): \theta^{-1}_t A=A\}$. A stationary ${\cal
I}$-trivial process is ergodic. Since ${\cal I}\subseteq {\cal T}$
(see Section 5.1 in \cite{thorisson}), a stationary ${\cal
T}$-trivial process (sometimes also called regular) is also ergodic.
The following version of Theorem 3.3 of Chapter 10 in
\cite{thorisson} states that an aperiodic positive recurrent
regenerative process can be successfully coupled with a stationary
ergodic process.
\begin{theorem}\label{key}
Let $(Z,S)$ be an aperiodic and positive recurrent regenerative
process. Let  $(Z^*,S^*)$ be a stationary version of it. Then the
following statements hold:
\begin{description}
  \item[a)] The space $(\Omega,{\cal F},\P)$ can be extended to support a
finite random time $T$ and a copy $Z'$ of $Z^*$ such that $(Z,Z',T)$
is a successful exact coupling of $Z$ and $Z^*$, i.e.
$$\theta_TZ=\theta_TZ',\quad\text{where}\quad Z'\buildrel{D}\over{=}
Z^*.$$
  \item[b)] The processes $Z$ and $Z'$ are  ${\cal T}$-trivial.
\end{description}
\end{theorem}
\begin{proof} The process $Z$ is aperiodic, which means that $T_1$ is a lattice with span 1.
Since $(Z,S)$ and $(Z^*,S^*)$ are discrete, the random variables
$S_0$ and $S^*_0$ are $\mathbb{Z}$-valued. So the assumptions of
Theorem 3.3 of Chapter 10 in \cite{thorisson} are fulfilled. The
claim {\bf a)} is claim {\bf a)} of that theorem, the ${\cal
T}$-triviality of $Z$ is claim {\bf d)} of that theorem. Finally,
the process $Z'$, being a stationary version of $Z$, is also an
aperiodic regenerative process with $S'_0$ being
$\mathbb{Z}$-valued. Hence it satisfies the same assumptions and is
therefore also
 ${\cal T}$-trivial.
\end{proof}
\begin{corollary}\label{cor:birkhoff}  Let $(Z,S)$ be an aperiodic and positive recurrent regenerative
process and  let  $(Z^*,S^*)$ be a stationary version of it. Let $g:
{\cal Z}^{\infty}\to \mathbb{R}$ be such that
$E|g(Z^*_1,Z^*_2,\ldots)|<\infty$. Then
\begin{equation}\label{birkhoff2} {1\over n}\sum_{t=1}^{n}
g(Z_t,Z_{t+1},\ldots)\to E[g(Z^*_1,Z^*_2,\ldots)]\quad  \text{  a.s.
and in}\,\, L_1.
\end{equation}
\end{corollary}
\begin{proof} Let us extend the space $(\Omega,{\cal F},\P)$ so
that the statements of Theorem \ref{key} hold. Then the process $Z'$
is stationary and ergodic having the same distribution as $Z^*$. By
Birkhoff's ergodic theorem then,
\begin{equation}\label{birkhoff}
{1\over n}\sum_{t=1}^{n} g(Z'_t,Z'_{t+1},\ldots)\to
E[g(Z'_1,Z'_2,\ldots)]=E[g(Z^*_1,Z^*_2,\ldots)]\quad \text{  a.s.
and in }L_1.\end{equation} Since the original process $Z$ can be
successfully coupled with $Z'$, it holds for almost every
realization of $Z$ and $Z'$ that they differ at the finite beginning
only. Since for a pathwise limit the beginning does not matter, we
immediately get the almost sure convergence of (\ref{birkhoff2}).
The $L_1$-convergence follows from applying Scheffe's lemma
separately to $g^+(Z_t,Z_{t+1},\ldots)$ and
$g^-(Z_t,Z_{t+1},\ldots)$.\end{proof}
\\\\
\noindent
{\bf Remark:} If $(Z,S)$ is positive recurrent but not aperiodic,
then Theorem \ref{key} cannot be applied. However, using Theorem 2.2
of \cite{thorisson} and noting that aperiodicity is not used in its
proof, a similar result can be obtained for   shift-coupling instead
of exact coupling. The process $Z'$ can be shown to be ${\cal
I}$-trivial and hence ergodic, thus Corollary \ref{cor:birkhoff}
still holds. In this paper we consider only aperiodic regenerative
processes.
\\\\
If $f: {\cal Z}\to \mathbb{R}$ is measurable, then  the convergence
(\ref{birkhoff2}) together with (\ref{jaotus}) yields
\begin{equation}\label{renewal}
{1\over n}\sum_{t=1}^n f(Z_t)\to Ef(Z^*_1)={1\over
ET_1}E\big[\sum_{t=1}^{T_1}f(Z^o_t)\big]={1\over
ET_1}E\big[\sum_{t=S_0+1}^{S_1}f(Z_t)\big]\quad \text{ a.s. and in }
L_1.\end{equation}
\subsection{Infinite Viterbi alignment}\label{subsec:viterbi}
\subsubsection{One-sided infinite Viterbi alignment}
\noindent {\bf Def.} Let for every $n$, $g^n: {\cal X}^n\to S^{n}$
be a classifier. We say that the sequence $\{g^n\}$ of classifiers
can be {\it extended to infinity}, if there exists a function $g:
{\cal X}^{\infty}\to S^{\infty}$ such that for almost every
realization $x^{\infty}\in {\cal X}^{\infty}$ the following
statement holds: for every $k\in \mathbb{N}$ there exists
$m(x^{\infty})\geq k$ such that for every $n\geq m$ the first $k$
elements of $g^n(x^n)$ are the same as the first $k$ elements of
$g(x^{\infty})$, i.e.~$g^n(x^n)_i=g(x^{\infty})_i,$ $i=1,\ldots,k$.
The function $g$
will be referred to as an {\it infinite alignment}.\\\\
If every observation is not classified independently, then the
existence of an infinite alignment is not trivial. It often happens
that adding one more observation $x_{n+1}$ changes the alignment
$g^n(x^n)$. This happens often with Viterbi or PMAP-alignments. The
existence of an infinite alignment allows to study asymptotic
properties of the alignment, which is usually done via the
corresponding {\it  alignment process}
$\{G_t\}_{t=1}^{\infty}:=g(X)$.
We consider the existence of infinite Viterbi alignments. Under
rather restrictive assumptions on HMMs the existence of an infinite
Viterbi alignment was first proved in \cite{caliebe1}. In
\cite{IEEE} it was proved under less restrictive assumptions. We now
introduce these assumptions and the corresponding results.
\\\\
Recall that $f_s$ are the densities of $P_s:=\P(X_1\in \cdot|Y_1=s)$
with respect to some reference measure $\mu$ on $({\cal X},{\cal
B})$. For each $s\in S$, let $G_s:=\{x\in\mathcal{X}: f_s(x)>0\}.$
We call a subset $C\subset S$  {\it a cluster} if the following
conditions are satisfied:
$$\min_{j\in C}P_j(\cap _{s\in C}G_s)>0 \quad{\rm and}\quad\max_{j\not\in C}P_j(\cap _{s\in C}G_s)=0.$$
Hence, a cluster is a maximal subset of states such that $G_C=\cap
_{s\in C}G_s$, the intersection of the supports of the corresponding
emission distributions, is  `detectable'. Distinct clusters need not
be disjoint and  a cluster can consist of a single state. In this
latter case such a state is not hidden, since it is exposed by any
observation it emits. If $|S|=2$, then $S$ is the only  cluster
possible, because otherwise the underlying Markov chain would cease
to be hidden.
The existence of $C$ implies the existence of a set $\X \subset
\cap_{s\in C }G_s$ and $\e>0$, $M<\infty$
 such that $\mu (\X)>0$, and $\forall x\in
{\X}$ the following statements hold: (i) $\e<\min_{s\in C} f_s(x)$;
(ii) $\max_{s\in C} f_s(x)<M$; (iii) $\max_{s\not \in C} f_s(x)=0$.
For proof, see \cite{IEEE}.
\\
The following two assumptions on HMMs are needed for the existence
of an infinite Viterbi alignment.\\\\
{\bf A1 (cluster-assumption)} There exists a cluster $C\subset S$
such that the sub-stochastic matrix $R=(P(i,j))_{i,j\in C}$ is
primitive, i.e.~there is a positive integer $r$ such that the $r$th
power of $R$ is strictly positive.\\
{\bf A2} For each state $l\in S$,
\begin{equation}\label{lll}
P_l\left(\left\{x\in\mathcal{X}:~f_l(x)p^*_{l}> \max_{s,s\ne
l}f_s(x)p^*_{s}\right\}\right)>0,\quad p^*_l=\max_jp_{j,l},\,\forall
l\in S.
\end{equation}
The cluster assumption {\bf A1} is often met in practice. It is
clearly satisfied if all elements of the matrix $P$ are positive.
Since any irreducible aperiodic matrix is primitive, the assumption
{\bf A1} is also satisfied if the densities $f_s$ satisfy the
following condition: for every $x\in {\cal X}$, $\min_{s\in
S}f_s(x)>0$, i.e.~for all $s\in S$, $G_s={\cal X}$. Thus, {\bf A1}
is more general than the {\it strong mixing condition} (Assumption
4.3.21 in \cite{HMMraamat}) and also weaker than Assumption 4.3.29
in \cite{HMMraamat}. Note that {\bf A1} implies the aperiodicity of
$Y$, but not vice versa.
The assumption {\bf A2} is more technical in nature. In \cite{K2} it
was shown that for a two-state HMM, (\ref{lll}) always holds for one
state, and this is sufficient for the infinite Viterbi alignment.
Hence, for the case $|S|=2$, {\bf A2} can be relaxed. Another
possibilities for relaxing {\bf A2} are discussed in \cite{AVT4,
IEEE}. To summarize: we believe that the cluster assumption {\bf A1}
is essential for HMMs, while the assumption {\bf A2}, although
natural and satisfied for many models, can be relaxed. For more
general discussion about these assumptions, see
\cite{AVT4,IEEE,PMAP,K2}.
\\
In the following, let  $\tilde{V}^n=v^n(X^n)$, where $v^n$ is a
finite Viterbi alignment. Let $U_t$ and $W_t$ be the stopping times
defined as
\begin{equation} \label{UtVt} W_t=\min\{\tau \ge t+r+1: X_{\tau-r}^{\tau} \in \mathcal{X}_0^{r+1}\}\, , \quad
 U_t=\max\{\tau \le t-r-1: X_{\tau}^{\tau+r} \in \mathcal{X}_0^{r+1}\}\, .\end{equation}
The results of the present paper are largely based on the following
theorem, which has been proved in \cite{IEEE,AVT4}. See also Lemma
2.1 in \cite{iowa}.
%
%
%
\begin{theorem}\label{viterbi} Let $(X,Y)=\{(X_t,Y_t)\}_{t=1}^{\infty}$ be a one-sided ergodic HMM
satisfying {\bf A1} and {\bf A2}. Then there exists an infinite
Viterbi alignment $v: {\cal X}^{\infty}\to S^{\infty}.$ Moreover,
the finite Viterbi alignments $v^n: {\cal X}^n\to S^n$ can be chosen
so that the following conditions are satisfied:
\begin{description}
 \item[R1] the process $Z:=(X,Y,V)$, where $V:=\{V_t\}_{t=1}^{\infty}$ is
the alignment process, is a positively recurrent aperiodic
regenerative process with respect to some renewal process
$\{S_t\}_{t=0}^{\infty}$;
  \item[R2] there exists an integer $m>0$ such that $S_0>m$ and
  \begin{itemize}
  \item[1)] for all $j\ge 0$ such that $S_j+m\le n$,
  $\V^n_t=V_t$ for all $t \le S_j$,
  \item[2)] $S_j-S_{j-1}\ge m$, $j=1,2,\ldots$;
  \end{itemize}
  \item[R3] the renewal times $\{S_k\}$ have the following property:
    \begin{itemize}
    \item[1)] if $S_k>t$, then $W_t\le S_k +m$,
    \item[2)] if $S_k<t$, then $U_t>S_k-m$.
    \end{itemize}
\end{description}
\end{theorem}
\begin{proof} The required infinite alignment is constructed
piecewise, see \cite{IEEE}. The regenerativity and positive
recurrence is shown in Section 4 of \cite{AVT4}. The aperiodicity
follows from the aperiodicity of $Y$ that follows from ${\bf A1}$.
The piecewise construction guarantees both {\bf R2} and {\bf R3}.
\end{proof}
\\\\
From now on we assume that the finite Viterbi alignments $v^n: {\cal
X}^n\to S^n$ are chosen according to Theorem \ref{viterbi}. These
choices of alignments are called {\bf consistent}. Obviously, the
consistent choice becomes an issue only if the finite Viterbi
alignment is not unique. In practice, the consistent choices can be
obtained just by predefined tie-breaking rules. With consistent
choices, the process $\tilde{Z}^n:=\{(\V^n_t,X_t,Y_t)\}_{t=1}^n$
satisfies by {\bf R2} the following property: $\tilde{Z}^n_t=Z_t$
for every $t=1,\ldots, S_{k(n)}$, where $k(n)=\max\{k\geq 0: S_k + m
\leq  n\}.$
\\\\
We now present a theorem that generalizes Theorem 3.1 of Chapter VI
in \cite{asmussen}. The proof is based on the same argument and
given in Appendix. Let $p\in \mathbb{N}$ and $g_p: {\cal Z}^p \to
\mathbb{R}$ be measurable. Define for every $i=p,\ldots,n$
$$\tilde{U}_i^n:=g_p(\tilde{Z}^n_{i-p+1},\ldots,\tilde{Z}^n_i).$$
If $i\leq S_{k(n)}$, then
$\tilde{U}_i^n={U}_i:=g_p(Z_{i-p+1},\ldots, Z_i).$ Finally, let
$$M_k:=\max_{S_{k}<n\leq S_{k+1}}|\tilde{U}_{S_k+1}^n+\cdots
+\tilde{U}_n^n|.$$ The random variables $M_p,M_{p+1},\ldots$ are
identically distributed, but for $p>1$ not necessarily independent.
Recall that $Z^*$ is a stationary version of $Z$.
%
%
%
\begin{theorem}\label{koond-r}
Let $g_p$ be such that $EM_p<\infty$ and
$E|g_p(Z^*_1,\ldots,Z_p^*)|<\infty$. Then
\begin{equation} \label{asmussen:sum}
{1\over n-p+1}\sum_{i=p}^n \tilde{U}^n_i\to E U_p=Eg_p(Z^*_1,\ldots,
Z^*_p)\quad \text{a.s. and in }L_1.\end{equation}
\end{theorem}
%
\subsubsection{Double-sided infinite Viterbi alignment}
\noindent {\bf Def.} Let for every $z_1,z_2\in \mathbb{Z}$,
$g_{z_1}^{z_2}: {\cal X}^{[z_1,z_2]}\to S^{[z_1,z_2]}$ be a
classifier. We say that the set $\{g_{z_1}^{z_2}\}$ of classifiers
can be {\it extended to infinity}, if there exists a function $g:
{\cal X}^{\mathbb{Z}}\to S^{\mathbb{Z}}$ such that for almost every
realization $x^{\infty}_{-\infty}\in {\cal X}^{\mathbb{Z}}$ the
following statement holds: for every $k\in \mathbb{N}$ there exists
$m\geq k$ (depending on $x_{-\infty}^{\infty}$) such that for every
$n\geq m$,
$$g_{-n}^{n}(x_{-n}^n)_i=g(x_{-\infty}^{\infty})_i,\quad
i=-k,\ldots,k.$$ The function $g$ will be referred to as an {\it
infinite double-sided alignment}.\\\\
The piecewise construction of the infinite Viterbi alignment allows
the double-sided extension as well.
\begin{theorem}\label{viterbi2} Let $(X,Y)=\{(X_t,Y_t)\}_{t=-\infty}^{\infty}$ be a double-sided ergodic HMM
satisfying {\bf A1} and {\bf A2}. Then there exists an infinite
Viterbi alignment $v: {\cal X}^{\mathbb{Z}}\to S^{\mathbb{Z}}.$
Moreover, the finite Viterbi alignments $v^{z_2}_{z_1}$ can be
chosen so that the following conditions are satisfied:
\begin{description}
  \item[RD1]  the process $(X,Y,V)$, where $V:=\{V_t\}_{t=-\infty}^{\infty}$ is
the alignment process, is a positively recurrent aperiodic
regenerative process with respect to some renewal process
$\{S_t\}_{t=-\infty}^{\infty}$;
  \item[RD2] there exists a nonnegative integer $m<\infty$ such that
  \begin{itemize}
    \item[1)] for every $j\geq 0$ such that $S_j+m\le n$, $\V^n_t=V_t$ for all $S_0\leq t \leq S_j;$
    \item[2)] $S_j-S_{j-1}\ge m$, $j \in {\mathbb{Z}}$;
  \end{itemize}
 \item[RD3] the renewal times $\{ S_k\}$ have the following property:
   \begin{itemize}
    \item[1)] if $S_k>t$, then $W_t\le S_k +m$,
    \item[2)] if $S_k<t$, then $U_t>S_k-m$;
   \end{itemize}
  \item[RD4] the mapping $v$ is a stationary coding, i.e.~$v(\theta(X))=\theta v(X)$,
  where $\theta$ is a shift operator:
  $\theta(\ldots,x_{-1},x_0,x_1,\ldots)=(\ldots,x_0,x_1,x_2,\ldots)$.
\end{description}
\end{theorem}
\begin{proof} The proof of {\bf RD1}, {\bf RD2} and {\bf RD3} is the same as in
Theorem \ref{viterbi}. Note the difference between {\bf R2} and {\bf
RD2}. The stationarity of $v$ follows from the fact that the
barriers in the construction of the infinite alignment are separated
(Lemma 3.2 in \cite{IEEE}).\end{proof}

\noindent In the following, the finite Viterbi alignments
$v_{z_1}^{z_2}$ are chosen to be consistent. The property {\bf RD4}
is important. Since $X$ is an ergodic process, from {\bf RD4} it
follows that the double-sided alignment process
$V=\{V_t\}_{t=-\infty}^{\infty}$ as well as the process
$\{(X_t,Y_t,V_t)\}_{t=-\infty}^{\infty}$ is an ergodic process. Let
$Z^*$ denote the restriction of
$\{(X_t,Y_t,V_t)\}_{t=-\infty}^{\infty}$  to the  nonnegative
integers, i.e. $Z^*=\{(X_t,Y_t,V_t)\}_{t=1}^{\infty}$. By {\bf RD2},
$Z^*$ is a stationary version of $Z$ as in {\bf R1}. Thus
$(X_0,Y_0,V_0)\buildrel{D}\over{=}(X^*_1,Y^*_1,V^*_1)=Z_1^*$ and we
shall often use this. Note that the one-sided Viterbi process $V$ in
{\bf R1} is not defined at time zero so that the random variable
$V_0$ always implies the double-sided, and hence stationary case.
\subsection{Smoothing probabilities}\label{subsec:exp}
\noindent Let $(X,Y)=\{(X_t,Y_t)\}_{t=-\infty}^{\infty}$ be a
double-sided HMM. From Levy's martingale convergence theorem it
immediately follows that for every state $j\in S$ and $z,t\in
\mathbb{Z}$, the limits of the smoothing probabilities
$\P(Y_t=j|X_z^{\infty}):=\lim_n\P(Y_t=j|X_z^n)$ and
$\P(Y_t=j|X_{-\infty}^{\infty}):=\lim_{z\to
-\infty}\P(Y_t=j|X_z^{\infty})$ exist almost surely. In \cite{PMAP}
it is shown that under {\bf A1} these probabilities satisfy the
following exponential forgetting inequalities:
\begin{align}\label{ineqrho1}
\|\P(Y_t\in \cdot|X_1^{\infty})-\P(Y_t\in
\cdot|X_{-\infty}^{\infty})\|&\leq C\rho^{t}\quad \rm{a.s. }\, ,\\
\label{ineqrho2} \|\P(Y_t\in \cdot|X_1^{\infty})-\P(Y_t\in
\cdot|X_{1}^{n})\|&\leq C\rho^{n-t}\quad \rm{a.s. },
\end{align}
where $C$ is a finite positive random variable, $\rho\in (0,1)$, in
the first inequality $t\geq 1$, and in the second inequality $n\geq
t\geq 1$. Here $\|\cdot\|$ stands for the total variation distance.
In what follows, we shall use the notation
$p_t(j|x_{-\infty}^{\infty}):=\P(Y_t=j|X_{-\infty}^{\infty}=x_{-\infty}^{\infty})$.
\section{Convergence of $R_1$-risk}\label{sec:r1}
\noindent Let the loss function be defined as in (\ref{point-loss})
and let $v^n$ be a consistently chosen Viterbi alignment. If the
underlying Markov chain would not be hidden, the {\it empirical risk
of the Viterbi alignment} could be directly calculated as follows:
\begin{equation}\label{true}
R_1(Y^n,X^n)={1\over n}\sum_{t=1}^n l(Y_t,v^n_t(X^n))={1\over
n}\sum_{t=1}^n l(Y_t,\V^n_t).
\end{equation}
The conditional expectation of $R_1(Y^n,X^n)$ given $X^n$ is the
random variable $R_1(v,X^n)=E[R_1(Y^n,X^n)|X^n].$
Since $S$ is finite and $l: S\times S \to \mathbb{R}$ is bounded,
from Theorem \ref{koond-r} and (\ref{renewal}) it follows that
\begin{equation}\label{risk-koond}
R_1(Y^n,X^n)\to El(Y_0,V_0)={1\over
ET_1}E\Big(\sum_{t=S_0+1}^{S_1}l(Y_t,V_t)\Big)=:R_1\quad \text{a.s.
and in }L_1.\end{equation} We shall call the constant $R_1$ {\it
asymptotic Viterbi risk}. It depends only on the model $(Y,X)$ and
on the loss function $l$. For $l(s,s')=I_{\{s'\ne s\}}$, the actual
risk is the average number of mistakes made by the Viterbi
alignment:
\begin{equation}\label{emp-vead}
R_1(Y^n,X^n)={1\over n}\sum_{t=1}^n I_{\{Y_t\ne
\V_t^n\}},\end{equation} and the corresponding asymptotic risk is
the asymptotic misclassification probability $\P(Y_0\ne V_0)$.
\\
To our knowledge, the idea of considering the $R_1$-type limits for
the Viterbi alignment has been first mentioned in \cite{Caliebe2},
the convergence of the empirical risk is also stated in \cite{iowa}.
To show the convergence of $R_1(v,X_n)$, we use the following lemma
(see Theorem 9.4.8 in \cite{Chung}).
\begin{lemma}\label{lemma:ting-koond}
Let $X_n$ be bounded random variables such that $X_n\to 0$ almost
surely. Let $\{{\cal F}_n\}_{n=1}^{\infty}$ be a filtration. Then
$E[X_n|{\cal F}_n]\to 0$ almost surely.
\end{lemma}
\noindent The following theorem is the first main result of this
paper. A similar result for the PMAP-alignment, namely the
convergence of $R_1(X^n)$ to a constant, is proved in \cite{PMAP}.
\begin{theorem}\label{main1} Let $\{(Y_t,X_t)\}_{t=1}^{\infty}$ be an ergodic
HMM satisfying {\bf A1} and {\bf A2}. Then there exists a constant
$R_1\geq 0$ such that the empirical risk and the risk of the Viterbi
alignment both converge to $R_1$ almost surely and in $L_1$:
\[ \label{kakskoond}
\lim_{n \to \infty} R_1(Y^n,X^n)=\lim_{n \to \infty}
R_1(v,X^n)=R_1\quad \text{a.s. and in } L_1. \]
Moreover, the expected risk of Viterbi alignments converges to $R_1$
as well: $ER_1(v,X^n)\to R_1$.\end{theorem}
\begin{proof} The convergence of the empirical risk is (\ref{risk-koond}). To show that $R_1(v,X^n)\to R_1$ a.s.,
apply Lemma \ref{lemma:ting-koond} with $X_n:=R_1(Y^n,X^n)-R_1$.
Clearly, $R_1(Y^n,X^n)-R_1$ is bounded and by (\ref{risk-koond}) it
goes to 0 a.s. Thus, by Lemma \ref{lemma:ting-koond},
$$|E[R_1(Y^n,X^n)-R_1|X^n]|=|E[R_1(Y^n,X^n)|X^n]-R_1|=|R_1(v,X^n)-R_1|\to
0\quad \text{a.s.}$$ By Scheffe's theorem, the convergence in $L_1$
follows by the non-negativity and boundedness of $R_1(v,X^n)$. The
convergence in $L_1$ implies the convergence of expected risks.
\end{proof}
\section{Convergence of $\R_1$-risk}\label{sec:logr1}
\noindent For the convergence of $\R_1$-risk we use Theorem
\ref{viterbi2}. Recall that the double-sided infinite alignment $v$
is a stationary coding. Consider the function $f: {\cal
X}^{\mathbb{Z}}\to (-\infty,0]$, where
$$f(x_{-\infty}^{\infty}):=\ln p_0\big(v(x_{-\infty}^{\infty}\big)_0|x_{-\infty}^{\infty})=\ln\P(Y_0=V_0|X_{-\infty}^{\infty}=x_{-\infty}^{\infty}).$$
In the following, let
$v_i(x_{-\infty}^{\infty}):=v(x_{-\infty}^{\infty})_i$ be the $i$-th
element of the infinite alignment. Note that for every
$t=1,2,\ldots$,
\begin{align*}
f\big(\theta_t(x_{-\infty}^{\infty}) \big)&=\ln
p_0\big(v_0(\theta_t(x_{-\infty}^{\infty}))\big|\theta_t(x_{-\infty}^{\infty})\big)
=\ln p_t\big(v_0(\theta_t(x_{-\infty}^{\infty}))\big|x_{-\infty}^{\infty}\big)\\
&=\ln
p_t\big(v_t(x_{-\infty}^{\infty})|x_{-\infty}^{\infty}\big)=\ln
\P(Y_t=V_t|X_{-\infty}^{\infty}=x_{-\infty}^{\infty}).
\end{align*}
Thus, by Birkhoff's ergodic theorem, there exists a constant $\R_1$
such that
\begin{equation}\label{r1stats}
-{1\over n}\sum_{t=1}^n\ln\P(Y_t=V_t|X_{-\infty}^{\infty})\to
-E\big(\ln\P(Y_0=V_0|X_{-\infty}^{\infty})\big)=:\R_1\quad
\text{a.s. and in  }L_1,\end{equation} provided the expectation is
finite.
The main idea for proving the convergence of $\R_1(v,X^n)$ is the
following. Consider without loss of generality a double-sided HMM
$\{(Y_t,X_t)\}_{t=-\infty}^{\infty}$. Then by {\bf RD2},
$\V_t^n=V_t$ for every $S_0\leq t \leq S_{k(n)}$, where
$k(n)=\max\{k\geq 0: S_k+m\leq n\}$ and $\{S_t\}_{t\geq 0}$ is the
renewal process as in Theorem \ref{viterbi2}. Thus,
$$-\frac{1}{n}\sum_{t=1}^n\ln\P(Y_t=\V^n_t|X^n)=-\frac{1}{n} \sum_{t=1}^{S_0-1}\ln\P(Y_t=\V^n_t|X^n)
-\frac{1}{n}\sum_{t=S_0}^{S_{k(n)}}\ln\P(Y_t=V_t|X^n)$$
\begin{equation} \label{R1partition}
-\frac{1}{n}\sum_{t=S_{k(n)}+1}^{n}\ln\P(Y_t=\V^n_t|X^n).
\end{equation}
The first term in the partition above converges to zero almost
surely. We will prove that the second term converges to $\R_1$
almost surely and that the third term converges to zero almost
surely. To prove the convergence of the second term, we need some
auxiliary results.
Let $C$ be the cluster as in {\bf A1} and let ${\cal X}_o$ be the
corresponding set. The proof of the following proposition is given
in Appendix.
%
%
\begin{proposition}\label{propvorr} Let $x_{-\infty}^{\infty}\in {\cal
X}^{\mathbb{Z}}$ be such that for some $u,v\in \mathbb{N}$,
$x_{-u}^{-u+r}\in {\cal X}_o^{r+1}$, $x_{v-r}^{v}\in {\cal
X}_o^{r+1}$ and for every $s\in S$, $\lim_n
p_0(s|x_{-n}^n)=p_0(s|x_{-\infty}^{\infty})$. Let
$v_0=v_0(x_{-\infty}^{\infty})$. Then there exist constants $c>0$
and $0<B<\infty$ that are independent of data such that
\begin{equation}\label{vaikeu}
p_0\big(v_0|x_{-\infty}^{\infty}\big)\geq
c\exp[-B(u+v)].\end{equation}
\end{proposition}
%
%
%
%
%
%
\noindent The proof of Proposition \ref{propvorr} reveals that it
holds also for a finite sequence of observations $x^n$. Moreover,
the following corollary holds.
\begin{corollary}\label{corrn} Let $x^n\in {\cal
X}^n$ be such that for some $w < n-r$, $x^{w+r}_w \in {\cal
X}_o^{r+1}$. Let  $\tilde{v}_t=v_t^n(x^n)$. Then there exist $c>0$
and $0<D<\infty$ such that for every $t$, $w<t\le n$,
\begin{equation}\label{vaiken}
p_t(\tilde{v}_t|x^n)\geq
    c\exp[-D(n-w)].
\end{equation}
\end{corollary}
\noindent The proof of Corollary \ref{corrn} follows the one of
Proposition \ref{propvorr} and is sketched in Appendix.
%
%
%
\begin{lemma}\label{lemma:kesk}
There exists $\alpha>0$ such that for every $t\in \mathbb{Z}$,
\begin{equation}\label{kesk2}
E\Big({1\over \P(Y_t=V_t|X_{-\infty}^{\infty})}\Big)^{\alpha}<\infty
\, .
\end{equation}
\end{lemma}
\begin{proof}
Let $W_0$ and $U_0$ be the stopping times defined in (\ref{UtVt}).
%
%
Because for every $s\in S$,
$\lim_n\P(Y_0=s|X_{-n}^n)=\P(Y_0=s|X_{-\infty}^{\infty})$ almost
surely, from (\ref{vaikeu}) it follows that
\begin{equation}\label{suurU}
\P(Y_0=V_0|X_{-\infty}^{\infty})\geq c\exp[-B(W_0-U_0)]\quad {\rm
a.s.}
\end{equation}
It holds that for some positive constants $a$ and $b$ and for every
$k=1,2,\ldots$,
$$\P(W_0>k)\leq a\exp(-bk),$$
see, e.g.~\cite{iowa}. This inequality implies that for $\alpha>0$
small enough, $E(e^{\alpha W_0})<\infty$. Analogously, for
sufficiently small $\alpha>0$, $E\big(e^{\alpha
(-U_0)}\big)<\infty$. Thus, by the Cauchy-Schwartz inequality it
holds that for sufficiently small $\alpha$,
\begin{equation}\label{summa}
E\big(e^{\alpha(W_0-U_0)}\big)=E\big(e^{\alpha W_0}e^{\alpha
(-U_0)}\big)\leq \Big(E\big(e^{2\alpha W_0}\big)E\big(e^{2\alpha
(-U_0)}\big)\Big)^{1\over 2}<\infty.\end{equation} The inequalities
(\ref{suurU}) and (\ref{summa}) imply (\ref{kesk2}) for $t=0$. By
the stationarity of $(X,Y)$, (\ref{kesk2}) holds for arbitrary
$t$.\end{proof}
\\\\
\noindent
Recall the inequalities (\ref{ineqrho1}) -- (\ref{ineqrho2}).
Unfortunately these bounds do not immediately hold for the
logarithms. The following lemma uses the inequality $|\ln a -\ln
b|\leq {1\over \min \{a,b\}}|a-b|$, provided that $a,b>0$.
\begin{lemma} \label{R1lemma} Suppose that for an $\alpha>0$,
\begin{equation}\label{kesk}
E\Big({1\over \P(Y_0=V_0|X_{-\infty}^{\infty})}\Big)^{\alpha}<\infty
\, .
\end{equation}
Then
\begin{equation}\label{koondr1}
\lim_{n\to \infty}\,-{1\over
n}\sum_{t=1}^{S_{k(n)}}\ln\P(Y_t=V_t|X^n) = \R_1 \quad {\rm a.s.}
\end{equation}
\end{lemma}
\begin{proof}
Let $\xi_t:=\P(Y_t=V_t|X_{-\infty}^{\infty})$,
$\eta^n_t:=\P(Y_t=V_t|X^n)$, $\eta_t:=\P(Y_t=V_t|X_{1}^{\infty})$
and let $\beta={1\over \alpha}$.
Take $m=n-(\ln n)^2$. Split the sum in (\ref{koondr1}) as
\[ -\frac{1}{n} \sum_{t=1}^{S_{k(n)}} \ln{\eta^n_t} =-\frac{1}{n} \sum_{t=1}^m \ln{\eta^n_t} -\frac{1}{n} \sum_{t=m+1}^{S_{k(n)}} \ln{\eta^n_t} =Term_I+Term_{II}\, . \]
We will prove that $Term_I$ converges to $\R_1$ and $Term_{II}$ to
zero almost surely.

{$\mathbf{Term_I}$}. Recall that $\{\xi_t\}$ is a stationary ergodic
process. The assumption (\ref{kesk}) ensures that $E|\ln
\xi_0|<\infty$.  Hence, by assumption,
$$\sum_{t=1}^{\infty}\P(\xi_t\leq {1\over t^{\beta}})=\sum_{t=1}^{\infty}\P(\xi^{-\alpha}_t
\geq {t})\leq E(\xi^{-\alpha}_t)+1<\infty\, .$$ Thus, the sequence
$\xi_t$, $t=1,2,\ldots $, satisfies $\P(\xi_t>{1\over
t^{\beta}}\quad {\rm ev})=1$. From (\ref{ineqrho1}) it follows that
$\P(\eta_t>{1\over 2t^{\beta}}\quad {\rm ev})=1$. Thus, almost
surely $|\ln \eta_t-\ln\xi_t|\leq C2t^{\beta}\rho^t$ eventually.
Since $-{1\over n}\sum_{t=1}^n\ln \xi_t\to \R_1$ almost surely, we
now have
\begin{equation}\label{lim1}
-{1\over n}\sum_{t=1}^n\ln \eta_t\to \R_1 \quad {\rm a.s.}
\end{equation}
%
%
\noindent Let (random) $T$ be so big that $\eta_t>{1\over
2t^{\beta}}$ when $t\geq T$. Observe that for $n$ large enough it
holds that
$$ -{\ln (4C)\over \ln \rho}-{\beta \over \ln \rho}\ln t \leq (\ln n)^2 \, .$$
Therefore, for large $n$ and $t$ such that $T<t\le n-(\ln n)^2$, we
have $C\rho^{(n-t)}\leq {1\over 4 t^{\beta}}$.
By (\ref{ineqrho2}), $|\eta^n_t-\eta_t|\leq C\rho^{n-t}$ almost
surely. Hence, for $n$ large enough and $t$ such that $T<t \le
n-(\ln n)^2$, $\min\{\eta_t,\eta_t^n\}\geq {1\over 4t^{\beta}}$ and
$|\ln \eta_t^n-\ln \eta_t|\leq (4t^{\beta}C)\rho^{n-t}$. Thus, as
$n\to \infty$,
\begin{align*}
\left|{1\over m}\sum_{t=1}^m \ln\eta_t^n-{1\over m}\sum_{t=1}^m
\ln\eta_t\right|&
\leq {1\over m}\sum_{t=1}^T |\ln\eta_t^n-\ln\eta_t|+{1\over
m}\sum_{t=T+1}^m (4t^{\beta}C)\rho^{n-t} \\ &\leq {1\over m}
\sum_{t=1}^T |\ln\eta_t^n-\ln\eta_t|+{4Cn\over m}
n^{\beta}\rho^{(\ln n)^2}\to 0 \quad {\rm a.s.}
\end{align*}
Since $m/n \to 1$, it follows from (\ref{lim1}) that $-{1\over
n}\sum_{t=1}^m \ln\eta_t^n \to \R_1$ almost surely.

{$\mathbf{Term_{II}}$}. It remains to prove that
\begin{equation}\label{kyssa}
-{1\over n}\sum_{t=m+1}^{S_{k(n)}} \ln \eta_t^n\to 0 \quad
\text{a.s.} \end{equation}
By Proposition \ref{propvorr}, $\P(Y_t=V_t|X_{-\infty}^{\infty})\geq
c\exp[-B(W_t-U_t)],$ where $U_t$ and $W_t$ are the stopping times
defined as in (\ref{UtVt}). Observe that when $S_1\leq t\leq
S_{k(n)}$, then according to {\bf RD3}, $U_t>0$ and $W_t\le
S_{k(n)}+m\le n$. Therefore, $U_t$ and $W_t$ are $X^n$-measurable
and for $S_1\leq t\leq S_{k(n)}$,
\[
E\left[\P(Y_t=V_t|X_{-\infty}^{\infty})|X^n\right]=\P(Y_t=V_t|X^n)\geq
c E\left[\exp[-B(W_t-U_t)] |X^n \right] = c\exp[-B(W_t-U_t)]\, .
\]
Thus for any $k$, $\P(-\ln \eta_t^n>k)\leq \P(B(W_t-U_t)>k+\ln
c)\leq a\exp[-bk],$ where the last inequality follows from
\cite{iowa}. Here $a$ and $b$ are positive constants. Since
\begin{align*}
\P\Big( -{1\over n}\sum_{t=m+1}^{S_k(n)}\ln \eta_t^n>\epsilon
\Big)&= \P\Big( \sum_{t=m+1}^{S_k(n)}-\ln \eta_t^n>n\epsilon\Big)
\leq \sum_{t=m+1}^{S_{k(n)}}\P\Big(-\ln \eta_t^n>{n\epsilon
\over(\ln
n)^2}\Big)\\
&\leq (\ln n)^2a\exp\Big[-b{ n\e\over(\ln n)^2}\Big]\end{align*} and
$$\sum_n (\ln n)^2a\exp\Big[-b{ n\e\over(\ln n)^2}\Big]<\infty,$$
the convergence in (\ref{kyssa}) follows by the Borel-Cantelli
lemma.
\end{proof}
%
%

\noindent We are now ready to prove the convergence of
$\R_1(v,X^n)$.
%
%
\begin{theorem}\label{main2} Let $\{(Y_t,X_t)\}_{t=1}^{\infty}$ be an ergodic
HMM satisfying {\bf A1} and {\bf A2}. Then there exists a constant
$\R_1$ such that
\[ \lim_{n\to \infty}\R_1(v,X^n)=\lim_{n\to \infty}\, -{1\over
n}\sum_{t=1}^n\ln\P(Y_t=\V^n_t|X^n) = \R_1 \quad {\rm a.s.}\,\,
\text{and in}\,\, L_1. \]
\end{theorem}
%
%
%
\begin{proof}
Consider the partition in ($\ref{R1partition}$). By Lemma
\ref{R1lemma}, the second term in (\ref{R1partition}) converges to
$\R_1$ almost surely. Thus, it suffices to prove that
\begin{equation}\label{lopp2}
{1\over n}\sum_{t=S_{k(n)}+1}^{n}\ln\P(Y_t=\V^n_t|X^n)\to 0\quad
{\rm a.s.}
\end{equation}
For every $k\geq 0$, let
$$M_k=\max_{S_k<n\leq S_{k+1}}|\ln
\P(Y_{S_k+1}=\V^n_{S_k+1}|X^n)+\cdots+\ln\P(Y_n=\V^n_{n}|X^n)|.$$
Because of {\bf R1}, for $S_k<n\le S_{k+1}$ and for $i$ such that
$S_k<S_k+i\le n$,
\[ \P(Y_{S_k+i}=\V^n_{S_k+i}|X^n)= \P(Y_{S_k+i}=\V^n_{S_k+i}|X_{S_k}^n)\, .\]
Therefore the random variables $M_k$ are i.i.d. As in the proof of
Theorem \ref{koond-r}, for (\ref{lopp2}) it suffices to show that
$EM_k<\infty$ for every $k\geq 0$, because then (\ref{lopp2})
follows due to the Borel-Cantelli lemma.
We shall consider $S_1$. The construction of  $S_k$ implies that
for every $k$, the observations $X_{S_k-m},\ldots,X_{S_k-m+r}$
belong to ${\cal X}_o$ (see \cite{IEEE}). Recall that we are
considering the case $n\leq S_2$. Hence, for every $t$ such that
$S_1<t\leq n$, by (\ref{vaiken}),
$$|\ln \P(Y_t=\V_t^n|X^n)|\leq D(n-S_1+m) +|\ln c| \leq D(S_2-S_1+m)+|\ln c|,$$
implying that $|M_1|\leq D(S_2-S_1+m)^2+(S_2-S_1) |\ln c|.$ The
renewal times  $S_2-S_1$ have all moments (see \cite{iowa, AVT4}),
hence $EM_1<\infty$.
\end{proof}
\\\\
\noindent {\bf Remark.} Note that the approach of the present
section can be easily applied to prove the convergence of the
$R_1$-risk:
 $R_1(v,X^n)\to R_1$ a.s. Indeed,
 the counterpart of (\ref{r1stats})  is
\[ \label{r11stats}
{1\over n}\sum_{t=1}^n\P(Y_t=V_t|X_{-\infty}^{\infty})\to
E\big(\P(Y_0=V_0|X_{-\infty}^{\infty})\big)=:1-R_1\quad \text{a.s.
and in }L_1.\] The inequalities (\ref{ineqrho1}) and
(\ref{ineqrho2}) immediately imply $$\lim_{n\to \infty}{1\over
n}\sum_{t=1}^n\P(Y_t=V_t|X^n) = 1-R_1 \quad {\rm a.s.},$$ and since
the probabilities are bounded, the convergence
$$R_1(v,X^n)=1-{1\over n}\sum_{t=1}^n\P(Y_t=\V_t^n|X^n)\to R_1 \quad {\rm a.s.}$$ now
easily follows. \\\\
From the remark above it is clear that the difficulties with the
$\R_1$-risk are due to unboundedness of $\ln \P(Y_t=\V_t^n|X^n)$,
since in principle $\P(Y_t=\V_t^n|X^n)$ can be arbitrarily small.
However, the latter is not so when instead of the Viterbi alignment
the PMAP-alignment is used. Then $\max_s\P(Y_t=s|X^n)\geq |S|^{-1}$.
By Birkhoff's theorem,
\begin{equation}\label{birkpmam}
-{1\over n}\sum_{t=1}^n\max_{s \in S}
\ln\P(Y_t=s|X_{-\infty}^{\infty})\to \R_1^*\quad \text{ a.s.  and
in  }L_1\, ,
\end{equation}
where $\R_1^*$ is a constant. The inequalities (\ref{ineqrho1}) and
(\ref{ineqrho2}) imply that
 $$|\max_s\ln\P(Y_t=s|X^n)-\max_s\ln\P(Y_t=s|X_{-\infty}^{\infty})|\leq {C |S|}(\rho^t+\rho^{n-t}) \quad \text{a.s.}$$
Thus, the convergence (\ref{birkpmam}) implies the convergence
\begin{equation}\label{koondpmap}
\R_1(X^n)=-{1\over n}\sum_{t=1}^n \max_{s\in S}\ln \P(Y_t=s|X^n)\to
\R_1^* \quad \text{ a.s.  and  in  }L_1.
\end{equation}
 Hence, the following corollary holds.
\begin{corollary}\label{maincorr}
There exists a constant $\R_1^*$ such that (\ref{koondpmap})
holds.\end{corollary}
\section{Convergence of $\R_{\infty}$-risk}\label{sec:loglike}
\noindent Recall that $\R_{\infty}(X^n)=-{1\over n}\ln
\P(Y^n=\V^n|X^n)$ and $\V^n=v^n(X^n)$. Let $p(x^n)$ be the
likelihood of $x^n$ and let $p(x^n|s^n)$ denote the conditional
likelihood of observing $x^n$ given that $\{Y^n=s^n\}$. Note that
$\ln p(x^n|s^n)$ can be expressed as
\begin{equation}\label{ln}
\ln p(x^n|s^n)=\sum_{t=1}^n\ln f_{s_t}(x_t)=\sum_{t=1}^n\ln
f_{1}(x_t)I_1(s_t)+\cdots + \sum_{t=1}^n\ln
f_{|S|}(x_t)I_{|S|}(s_t).
\end{equation}
To prove the convergence of $\R_{\infty}(X^n)$, write
$\P(Y^n=\V^n|X^n)$ as
$$\P(Y^n=\V^n|X^n)={p(X^n|\V^n)\P(Y^n=\V^n)\over p(X^n)}.$$
Then
\begin{equation} \label{split-likelihood}
\R_{\infty}(X^n)=-{1\over n}\Big(\ln p(X^n|\V^n)+\ln
\P(Y^n=\V^n)-\ln p(X^n) \Big).\end{equation}
Before stating the theorem about the convergence of
$\R_{\infty}(X^n)$, we introduce the conditional measure
$Q_s:=\P(X_0\in \cdot|V_0=s)$, $s\in S$. As it follows from Theorem
\ref{koond-r}, the measure $Q_s$ is the almost sure limit of the
empirical measure corresponding to the Viterbi alignment state $s$,
i.e.~for every Borel set $A$,
\[\label{q-moot}
{\sum_{t=1}^n I_{A\times s}(X_t,\V^n_t)\over \sum_{t=1}^n
I_{s}(\V^n_t)}\to Q_s(A)\quad \text{a.s.}
\]
This convergence is the basis of the adjusted Viterbi training
introduced in \cite{AVT1,AVT4}. Note that for every $Q_s$-integrable
$g$,
\begin{equation}\label{g}
E\big(g(X_0)I_s(V_0)\big)=E\big(g(X_0)|V_0=s\big)\P(V_0=s)=m_s\!\!\int
\!\! g(x)Q_s(dx),
\end{equation}
where $m_s:=\P(V_0=s)$.
\begin{theorem}\label{main3} Let for every $s\in S$ the logarithm of the conditional density $f_s$ be $P_s$-integrable. Then
\[\label{likelihood2}
-\R_{\infty}(X^n)\to \sum_{s\in S} m_s \!\!\int \!\! \ln
f_s(x)Q_s(dx)+ E[\ln p_{V^*_1 V^*_2}]+H_X=: -\R_{\infty}\quad
\rm{a.s}\text{ and in }L_1,\]
where $H_X$ is the entropy rate of $X$ and
$p_{ij}=\P(Y_2=j|Y_1=i)$.\end{theorem}
\begin{proof}
Consider (\ref{split-likelihood}). To prove the convergence of the
first term of the RHS, apply (\ref{ln}) to the Viterbi alignment. In
\cite{AVT3} it was shown that if $\ln f_s$ is $P_s$-integrable, then
$\ln f_s$ is also $Q_s$-integrable for every $s$. Then by Theorem
\ref{koond-r} and (\ref{g}), for every state $s\in S$
\[\label{f}
{1\over n}\sum_{t=1}^n \ln f_s(X_t)I_s(\V^n_t)\to E\big( \ln
f_s(X_0)I_s(V_0)\big)=m_s \!\!\int \!\! \ln f_s(x)Q_s(dx)\quad
\rm{a.s.}\text{ and in } L_1.\]
This together with (\ref{ln}) gives
\[\label{likelihood1}
{1\over n}\ln p(X^n|Y^n=v^n(X^n))\to \sum_{s\in S} m_s \!\!\int \!\!
\ln f_s(x)Q_s(dx)\quad \text{a.s.}\text{ and in } L_1.\]
For the second term use the Markov property
\begin{align*}
\ln \P(Y^n=\V^n)&=\ln \pi_{\V^n_1}+\ln p_{\V^n_1 \V^n_2 }+\cdots +
\ln p_{\V^n_{n-1} \V^n_n},\end{align*} where  $\pi_s=\P(Y_1=s)$.
Since $\V^n$ is a path with positive likelihood,
$p_{\V^n_t,\V^n_{t+1}}>0$ almost surely for every $t$. Because the
number of states is finite, there exists a constant $M>0$ such that
for every $i$, $-\ln p_{\V^n_i \V^n_{i+1}}< M$ almost surely. Hence
the assumptions of Theorem \ref{koond-r} hold and, with $p_{\V^n_0
\V^n_1}=\pi_{\V^n_1}$, we get
\begin{align*}
{1\over n}\ln \P(Y^n=\V^n)={1\over n}\sum_{t=0}^{n-1} \ln p_{\V^n_t
\V^n_{t+1}}\to E[\ln p_{V^*_1 V^*_2}] \quad \text{a.s and in }L_1,
\end{align*}
where $E[\ln p_{V^*_1 V^*_2}]=\sum_{i,j\in S}\ln
p_{ij}\P(V^*_1=i,V^*_2=j)$. Finally, the Shannon-McMillan-Breiman
theorem implies the convergence of the third term of the RHS in
(\ref{split-likelihood}):
$${1\over n}\ln p(X^n)\to -H_X\quad \text{a.s. and in }L_1.$$
\end{proof}
\\\\
\noindent {\bf Remark.} Note that $-E[\ln p_{Y_1 Y_2}]$ is the
entropy rate of $Y$.  By the same argument,
\[ \label{markov-log}
{1\over n}\ln \P(Y^n|X^n)\to \sum_{s \in S} \pi_s\int \ln
f_s(x)P_s(dx)-H_Y+H_X=: -\R^Y_{\infty}\quad \text{a.s.
 and  in   }L_1,\]
where $H_Y$ is the entropy rate of $Y$. The convergence in $L_1$
implies
$$-{1\over n}E[\ln \P(Y^n|X^n)]\to \R^Y_{\infty},$$
where the expectation is taken over $X^n$ and $Y^n$. Since $E[\ln
\P(Y^n|X^n)]=H(Y^n|X^n)$ (the conditional entropy of $Y^n$ given
$X^n$), the limit $\R^Y_{\infty}$ could be interpreted as the
conditional entropy rate of $Y$ given $X$, it is not the entropy
rate of $Y$. Clearly, $\R_{\infty}\leq \R^Y_{\infty}$,
and the difference of those two numbers shows how much the Viterbi
alignment "overestimates" the likelihood.
\appendix
\section{Proofs of Theorem \ref{koond-r}, Proposition
\ref{propvorr} and Corollary \ref{corrn}}
%
\subsection{Proof of Theorem \ref{koond-r}}
\begin{proof} Partition the sum in (\ref{asmussen:sum}) as
$${1\over n-p+1}\sum_{i=p}^n \tilde{U}^n_i={1\over
n-p+1} \Big( \sum_{i=p}^{S_{k(n)}}{U}_i+\sum_{i=S_{k(n)}+1}^n
\tilde{U}^n_i\Big).$$ Since
 $S_{k(n)}\nearrow  \infty$ almost surely, from (\ref{birkhoff2}) we know that
\begin{equation}\label{u-koond1}
{1\over S_{k(n)}}\sum_{i=p}^{S_{k(n)}} {U}_i \to Eg_p(Z^*_1,\ldots,
Z^*_p) \quad \text{a.s. and in }L_1.
\end{equation}
Since $ET_1<\infty$ and $n\geq p$, by SLLN and the elementary
renewal theorem
$${S_{k(n)}\over n-p+1}={S_{k(n)}\over k(n)}{k(n)\over n-p+1}\to 1 \quad \text{a.s. and in }L_1.$$
Combining this with (\ref{u-koond1}) and taking into account that
the sequence $\{{S_{k(n)}\over n-p+1}\}$ is bounded, we obtain that
$$
{1\over n-p+1}\sum_{i=p}^{S_{k(n)}}{U}_i\to Eg_p(Z^*_1,\ldots,
Z^*_p)\quad \text{a.s. and in}\, L_1.$$ Note that
$$
\Big|{1\over n-p+1} \sum_{i=S_{k(n)}+1}^n \tilde{U}^n_i\Big|\leq
{M_{k(n)}\over {S_{k(n)}+1-p}}\leq {M_{k(n)}\over {{k(n)}-p+1}}.$$
%
%
%
Since the random variables $M_k$, $k\geq p$, are indentically
distributed, it holds for every $\epsilon>0$ that
$$\sum_{k=p}^{\infty}\P\Big({M_k\over k}>\epsilon\Big )=
\sum_{k=p}^{\infty}\P\Big ({M_p\over \epsilon}>k \Big)\leq
{EM_p\over \epsilon}<\infty\, .$$ Thus, by the Borel-Cantelli lemma
${M_k\over k}\to 0$ almost surely as $k\to \infty$. Clearly,
$E\left[{M_k\over k}\right]\to 0$, so by Scheffe's theorem
${M_k\over k}\to 0$ in $L_1$ as well.
\end{proof}
\subsection{Preliminaries for proving Proposition \ref{propvorr} and Corollary \ref{corrn}}
\noindent Let us start with some notation. For every sequence of
observations $x_k^l=(x_k,\ldots,x_l)\in {\cal X}^{l-k+1}$, for every
sequence of states $y_k^l=(y_k,\ldots,y_l)\in {S}^{l-k+1}$ and
states $i,j\in S$, we denote by $p(x_k^l,y_k^l,j|i)$ the following
conditional likelihood:
$$p(x_k^l,y_k^l,j|i):=P(i,y_k)\prod_{u=k}^{l-1}P(y_u,y_{u+1})P(y_l,j)\prod_{u=k}^lf_{y_u}(x_u).$$
Similarly,
$$p(x_k^l,y_k^l|i):=\sum_jp(x_k^l,y_k^l,j|i), \quad p(x_k^l,y_k^l):=\sum_ip(x_k^l,y_k^l|i)\pi(i).$$
We also define
$$\alpha(x_k^l,s):=\sum_{y_k^l\in S^{k-l+1}:y_l=s}p(x_k^l,y_k^l),\quad
\beta(x_k^l|i)=\sum_{y_k^l\in S^{k-l+1}}p(x_k^l,y_k^l|i).$$ The last
two notations are standard in the HMM literature, see e.g.
\cite{HMP,HMMraamat}. Let
$$\beta(x_k^l,s|i)=\sum_{y_k^l \in S^{k-l+1}: y_l=s}p(x_k^l,y_k^l|i),\quad  \alpha(s,x_k^l):=\sum_{y_k^l\in S^{k-l+1}:y_k=s}p(x_k^l,y_k^l).$$
Finally, let
$$\sigma(x_k^l,j|i):=\max_{y_k^l}p(x_k^l,y_k^l,j|i),\quad \sigma(x_k^l|i):=\max_{y_k^l}p(x_k^l,y_k^l|i).$$
%
%
%
Let $C$ be the cluster as in {\bf A1}. Thus, there is an $r\geq 1$
such that the matrix  $R^r$ has positive entries. Let ${\cal X}_o$
be the corresponding set. Suppose $z^{r}\in {\cal X}_o^{r}$ and
$y^{r}\in C^{r}$. By the definition of $\X$,  it holds that
$${\e}^{r}\leq \big(\prod_{u=1}^{r}f_{y_u}(z_u)\big)\leq M^{r}.$$
By the cluster assumption,  $ 0<\min_{i,j\in C}R^r(i,j)\leq
\big(P(i,y_1)P(y_1,y_2)\ldots P(y_{r-1},j)\big)\leq 1,$ provided
$i,j\in C$. Hence there exist constants $0<a<A<\infty$, not
depending on the observations, such that
\begin{equation}\label{aA}
a<p(x^r,y^r|i)<A\quad \text{and}\quad
a<p(x^{r-1},y^{r-1},j|i)<A,\quad j\in C.
\end{equation}
Suppose now $x^m$, $m>r$, is a sequence of observations such that
the first $r$ elements belong to the set $\X$, i.e.~$x^r\in \X^r$.
Then for every $i$, $p(x^m,y^m|i)>0$ only if $y^r\in C^r$, implying
that
$$\sigma(x^m,j|i)=\max_{s\in C}\max_{y^r\in C^r:
y_r=s}p(x^r,y^r|i)\sigma(x_{r+1}^m,j|s).$$ Let now $i_1,i_2\in C$.
Then for some states $s_1,s_2\in C$,
\begin{align*}
\sigma(x^m,j|i_1)&=\max_{y^r\in C^r:
y_r=s_1}p(x^r,y^r|i_1)\sigma(x_{r+1}^m,j|s_1),\\
\sigma(x^m,j|i_2)&=\max_{y^r\in C^r:
y_r=s_2}p(x^r,y^r|i_2)\sigma(x_{r+1}^m,j|s_2)\geq \max_{y^r\in C^r:
y_r=s_1}p(x^r,y^r|i_2)\sigma(x_{r+1}^m,j|s_1).\end{align*} Hence,
the inequalities (\ref{aA}) imply that for every state $j$
\begin{equation}\label{ratio1}
{\sigma(x^m,j|i_1)\over \sigma(x^m,j|i_2)}\leq {\max_{y^r\in C^r:
y_r=s_1}p(x^r,y^r|i_1)\over \max_{y^r\in C^r:
y_r=s_1}p(x^r,y^r|i_2)}\leq {A\over a}.\end{equation}
Similarly, if $x^m$ is such that the last $r$ elements belong to
${\cal X}_o$, i.e. $x_{m-r+1}^m\in {\cal X}^r$, then for arbitrary
states $j_1,j_2\in C$ there exist $s_1,s_2\in C$ such that
\begin{align*}
\sigma(x^m,j_1|i)&=\max_{y^{m-r+1}:
y_{m-r+1}=s_1}p(x^{m-r+1},y^{m-r+1}|i)\sigma(x_{m-r+2}^m,j_1|s_1),\\
\sigma(x^m,j_2|i)&=\max_{y^{m-r+1}:
y_{m-r+1}=s_2}p(x^{m-r+1},y^{m-r+1}|i)\sigma(x_{m-r+2}^m,j_2|s_2)\\
&\geq \max_{y^{m-r+1}:
y_{m-r+1}=s_1}p(x^{m-r+1},y^{m-r+1}|i)\sigma(x_{m-r+2}^m,j_2|s_1).
\end{align*}
So from (\ref{aA}) it follows that
\begin{equation}\label{ratio2}
{\sigma(x^m,j_1|i)\over \sigma(x^m,j_2|i)}\leq
{\sigma(x_{m-r+2}^m,j_1|s_1)\over  \sigma(x_{m-r+2}^m,j_2|s_1)}\leq
{A\over a}.\end{equation}
%
\subsection*{Proof of Proposition \ref{propvorr}}
%
\begin{proof} Let $x_{-\infty}^{\infty}$ be a sequence of observations and
let $x_{-n}^n$ be its subword. For every state $i\in S$, we are
interested in probability
$p_0(i|x_{-n}^n):=\P(Y_0=i|X_{-n}^n=x_{-n}^n)$. Note that
$$p_0(i|x_{-n}^n)p(x_{-n}^n)=\sum_{y_{-n}^n:y_0=i}p(x_{-n}^n,y_{-n}^n)=:\gamma_0(x_{-n}^n,i).$$
Observe that for every $u,v\in \{1,\ldots,n-1\}$ and for an
arbitrary state, let it be 1,
\begin{align*}
\gamma_0(x_{-n}^n,i) &=\sum_{s_1\in S}\sum_{s_2\in S}\sum_{s_3\in
S}\sum_{s_4\in
S}\alpha(x_{-n}^{-u},s_1)\beta(x_{-u+1}^{-1},s_2|s_1)P(s_2,1)f_1(x_0)\beta(x_{1}^{v-1},s_3|1)P(s_3,s_4)\alpha(s_4,x_{v}^n)\\
&\geq \sum_{s_1\in S}\sum_{s_4\in
S}\alpha(x_{-n}^{-u},s_1)\sigma(x_{-u+1}^{-1},1|s_1)f_1(x_0)\sigma(x_{1}^{v-1},s_4|1)\alpha(s_4, x_{v}^n)\\
&\geq
p(x_{-n}^{-u})\big(\min_{s}\sigma(x_{-u+1}^{-1},1|s)\big)f_1(x_0)
\big(\min_{s}\sigma(x_{1}^{v-1},s|1)\big)p(x_{v}^n).\end{align*}
Without loss of generality assume $v_0(x_{-\infty}^{\infty})=1$. Let
$v_{-u}(x_{-\infty}^{\infty})=a$ and
$v_{v}(x_{-\infty}^{\infty})=b$. By Bellman's optimality principle,
for every $i_o\in S$
$$\sigma(x_{-u+1}^{-1},1|a)f_1(x_0)\sigma(x_1^{v-1},b|1)\geq
\sigma(x_{-u+1}^{-1},i_o|a)f_{i_o}(x_0)\sigma(x_1^{v-1},b|i_o),$$
implying that for every state $i_o$,
$$f_1(x_0)\geq {\sigma(x_{-u+1}^{-1},i_o|a)\over \sigma(x_{-u+1}^{-1},1|a)} f_{i_o}(x_0){\sigma(x_1^{v-1},b|i_o)\over \sigma(x_1^{v-1},b|1)}.$$
Thus,
\begin{align}\label{alumine}
\gamma_0(x_{-n}^n,1)\geq
p(x_{-n}^{-u}){\big(\min_{s}\sigma(x_{-u+1}^{-1},1|s)\big)\over
\sigma(x_{-u+1}^{-1},1|a)}\sigma(x_{-u+1}^{-1},i_o|a)f_{i_o}(x_0)\sigma(x_1^{v-1},b|i_o)
{\big(\min_{s}\sigma(x_{1}^{v-1},s|1)\big) \over
\sigma(x_1^{v-1},b|1)}p(x_{v}^n) .\end{align}
Note that for every $x_k^m$,
$$\sum_s\beta(x_k^m,s|i)P(s,j)=\sum_{y_k^m}p(x_k^m,y_k^m,j|i)\leq
|S|^{m-k+1}\sigma(x_k^m,j|i).$$
Therefore, for every $i_o \in S$
\begin{align*}
\gamma_0(x_{-n}^n,i_o)&=\sum_{s_1\in S}\sum_{s_2\in S}\sum_{s_3\in
S}\sum_{s_4\in
S}\alpha(x_{-n}^{-u},s_1)\beta(x_{-u+1}^{-1},s_2|s_1)P(s_2,i_o)f_{i_o}(x_0)
\beta(x_1^{v-1},s_3|i_o)P(s_3,s_4)\alpha(s_4,x_{v}^n)\\
&\leq \sum_{s_1\in S}\sum_{s_4\in
S}\alpha(x_{-n}^{-u},s_1)|S|^{u-1}\sigma(x_{-u+1}^{-1},i_o|s_1)f_{i_o}(x_0)|S|^{v-1}\sigma(x_{1}^{v-1},s_4|i_o)\alpha(s_4,x_{v}^n)\\
&\leq p(x_{-n}^{-u})|S|^{u-1}\big(\max_{s\in
S}\sigma(x_{-u+1}^{-1},i_o|s)\big)f_{i_o}(x_0)|S|^{v-1}\big(\max_{s\in
S}\sigma(x_{1}^{v-1},s|i_o)\big)p(x_{v}^n).
\end{align*}
Let $x_{-n}^n$ be such that $x_{-u}^{-u+r}\in {\cal X}_o^{r+1}$ and
$x_{v-r}^{v}\in {\cal X}_o^{r+1}$. Then $\alpha(x_{-n}^{-u},s_1)=0$
if $s_1\not\in C$, since $x_{-u}\in {\cal X}_o$. Analogously,
$\alpha(s_4,x_{v}^n)=0$ if $s_4\not\in C$. Thus, in this case the
inequality above becomes
\[\label{ülemine}
\gamma_0(x_{-n}^n,i_o)\leq p(x_{-n}^{-u})|S|^{u-1}\big(\max_{s\in
C}\sigma(x_{-u+1}^{-1},i_o|s)\big)f_{i_o}(x_0)|S|^{v-1}\big(\max_{s\in
C}\sigma(x_{1}^{v-1},s|i_o)\big)p(x_{v}^n).\] The same holds for
(\ref{alumine}), implying that
\begin{align*}
{\gamma_0(x_{-n}^n,1)\over \gamma_0(x_{-n}^n,i_o)}\geq & {\min_{s\in
C}\sigma(x_{-u+1}^{-1},1|s)\over
\sigma(x_{-u+1}^{-1},1|a)}{\sigma(x_{-u+1}^{-1},i_o|a) \over
\max_{s\in C}\sigma(x_{-u+1}^{-1},i_o|s)}\times\\
&\times  {\sigma(x_1^{v-1},b|i_0)\over \max_{s\in
C}\sigma(x_{1}^{v-1},s|i_o)}{\min_{s\in C}\sigma(x_{1}^{v-1},s|1)
\over \sigma(x_1^{v-1},b|1)}|S|^{2-(u+v)}.\end{align*}
The inequalities (\ref{ratio1}) and (\ref{ratio2}) imply that the
ratios above are bounded below by ${a\over A}$ that does not depend
on the observations. Thus, there exist constants $c_1$  and
$0<B<\infty$ (not depending on the data) such that for every state
$i_o$,
\begin{equation}\label{kokku}
{p_0(1|x_{-n}^n)\over p_0(i_o|x_{-n}^n)}= {\gamma_0(x_{-n}^n,1)\over
\gamma_0(x_{-n}^n,i_o)}\geq c_1\exp[-B(u+v)].\end{equation} Since
$\sum_{i\in S}p_0(i|x_{-n}^n)=1$, there exists $i_o$ such that
$p_0(i_o|x_{-n}^n)\geq |S|^{-1}$. Thus, by (\ref{kokku}),
$$p_0(1|x_{-n}^n)\geq {c_1\over |S|}\exp[-B(u+v)].$$
Because $p_0(1|x_{-n}^n)\to p_0(1|x_{-\infty}^{\infty})$, the
inequality (\ref{vaikeu}) follows by taking $c={c_1\over |S|}$.
\end{proof}
\subsection*{Proof of Corollary \ref{corrn}}
\begin{proof}
The proof is analogous to the proof of Proposition \ref{propvorr}.
Using the same notations we obtain that for every $t$, $w < t< n$,
\[ \gamma_t(x^n,\tilde{v}_t)\ge
 p(x^{w})\big(\min_{s \in C}\sigma(x_{w+1}^{t-1},\tilde{v}_t|s)\big)f_{\tilde{v}_t}(x_t)\sigma(x_{t+1}^n|\tilde{v}_t).
\]
For every $i_o\in S$,
\[
\gamma_t(x^n,i_o)\le
 p(x^{w})\big(\max_{s \in C}\sigma(x_{w+1}^{t-1},i_o|s)\big)f_{i_o}(x_t)\sigma(x_{t+1}^n|i_o)|S|^{n-w-1}.
\]
Let $v_w(x^n)=b$. By Bellman's optimality principle,
\[
f_{\tilde{v}_t}(x_t)\geq {\sigma(x_{w+1}^{t-1},i_o|b)\over
\sigma(x_{w+1}^{t-1},{\tilde{v}_t}|b)}
f_{i_o}(x_t){\sigma(x_{t+1}^{n}|i_o)\over
\sigma(x_{t+1}^{n}|{\tilde{v}_t})}.
\]
Thus,
\[{p_t({\tilde{v}_t}|x^n)\over p_t(i_o|x^n)}= {\gamma_t(x^n,{\tilde{v}_t})\over
\gamma_t(x^n,i_o)}\ge {\min_{s\in
C}\sigma(x_{w+1}^{t-1},{\tilde{v}_t}|s)\over
\sigma(x_{w+1}^{t-1},{\tilde{v}_t}|b)}{\sigma(x_{w+1}^{t-1},i_o|b)
\over \max_{s\in C}\sigma(x_{w+1}^{t-1},i_o|s)} |S|^{-(n-w-1)}.
\]
Because the ratios above are bounded below by ${a\over A}$ and
$p_t(i_o|x^n)\geq |S|^{-1}$ for some $i_o\in S$, the statement of
the corollary follows with $D=\ln |S|$.
\end{proof}

\bibliographystyle{plain}      

\bibliography{kuljus(dets2010)}

\end{document}